\documentclass[a4paper,11pt]{amsart}

\usepackage{amssymb} \usepackage{fullpage} \usepackage[OT4]{fontenc}

\theoremstyle{definition}
\newtheorem{definition}{Definition}[section]
\theoremstyle{plain} \newtheorem{theorem}[definition]{Theorem}
 \newtheorem{lemma}[definition]{Lemma}
 \newtheorem{corollary}[definition]{Corollary}
 \theoremstyle{remark} \newtheorem{remark}[definition]{Remark}
\newtheorem{example}[definition]{Example}

\DeclareMathOperator{\mult}{mult}
\DeclareMathOperator{\ord}{ord}

\def\field{\mathbb{C}} \def\PP{\mathbb{P}} \def\ZZ{\mathbb{Z}} 
\def\Z{\mathbb{Z}}
\def\P{\mathbb{P}}
\def\C{\mathbb{C}}
\def\calo{\mathcal{O}}
\newcommand\wtilde[1]{\widetilde{#1}}
\newcommand\subm{\underline{m}}

\newcommand\submj{\underline{m^{(j)}}}
\newcommand\subnj{\underline{n^{(j)}}}
\newcommand\mj{m^{(j)}}
\newcommand\nj{n^{(j)}}
\newcommand\eps{\varepsilon}

 \let\to\longrightarrow  

\begin{document}

\title{Points fattening on $\PP^1 \times \PP^1$ and symbolic powers of bi-homogeneous ideals}

\author{Magdalena Baczy\'nska}
\author{Marcin Dumnicki}
\author{Agata Habura}
\author{Grzegorz Malara}
\author{Piotr Pokora}
\author{Tomasz Szemberg}\thanks{Szemberg research was partially supported
by NCN grant UMO-2011/01/B/ST1/04875}
\author{Justyna Szpond}
\author{Halszka Tutaj-Gasi\'nska}




\begin{abstract}
   We study symbolic powers of bi-homogeneous ideals of points in $X=\PP^1\times \PP^1$
   and extend to this setting results on the effect of points fattening
   obtained in \cite{BocCha11} and \cite{DST13}.
   We prove a Chudnovsky-type theorem for bi-homogeneous
   ideals and apply it to classification
   of configurations of points with minimal or no fattening effect.
\end{abstract}

\maketitle

\section{Introduction}
   The study of the effect of the points fattening on postulation
   in $\P^2$ was initiated by Bocci and Chiantini in \cite{BocCha11}.
   For a homogeneous ideal $I\subset\C[\P^n]$ its \emph{initial degree} $\alpha(I)$ is defined
   as the least integer $t$ such that the
   graded part $I_t$ is non-zero.
   Let $I$ be the radical ideal of a set of points $Z\subset\P^2$.
   Bocci and Chiantini asked how passing from $Z$ to the double
   scheme structure $2Z$ (this is the fattening mentioned in the title)
   changes the initial degrees of the associated ideals (this is the effect
   of fattening mentioned in the title, the bigger the difference
   $\alpha(2Z)-\alpha(Z)$, the bigger the effect).
   By the classical Nagata-Zariski theorem \cite[Theorem 3.14]{Eis95}
   the ideal of $2Z$ is the second symbolic power of $I$.
   Of course the $m$--fold structure $mZ$ is defined by $I^{(m)}$ for all $m\geq 1$.
   In \cite{DST13} three of the authors of the present note extended
   Bocci-Chiantini analysis to arbitrary symbolic powers
   of radical ideals of point configurations in $\P^2$.

   The purpose of this note is to study analogous questions
   for $X:=\P^1\times\P^1$.
   This might appear at the first glance as a minor modification,
   yet some new phenomena appear and necessary modifications
   when compared with $\P^2$ indicate how similar problems
   could be studied on arbitrary surfaces.

   Since the ideals under consideration are now bi-homogeneous
   the choice of the right extension of the initial degree notion
   is more facultative. Somehow intuitively, given a set of points $Z\subset X$,
   this should be the least bi-degree of a curve passing through $Z$.
   But which bi-degree is the smallest? We propose here two natural variants
   of
   answering this question, both leading to some nice geometrical
   consequences, see Definition \ref{def:2 alphas} for details.
   In both cases we give a fairly complete classification
   of configurations of points with relatively small
   effect of fattening. The main results of this note are
   Theorem \ref{th:t1}, Theorem \ref{thm:jumps by 1}
   and Theorem \ref{thm:jumps by 1 alpha plus}. Theorem \ref{Chu}
   generalizes Chudnovsky theorem
   on polynomial interpolation in $\P^2$ \cite[General Theorem 6]{Chu81}
   to the bi-homogeneous setting and could be of independent interest.

   The $\alpha^*$ invariant introduced in Definition \ref{def:2 alphas}
   is related to the anti-canonical divisor on $X$. As such
   it can be considered on arbitrary del Pezzo surfaces. With
   little adjustments the questions investigated here can be studied
   on arbitrary polarized surfaces
   (in fact even on varieties of arbitrary dimension). We hope to come back to that
   in the near future.
\subsection{Set-up and notation}
   Throughout the paper we denote by $X$ the Cartesian product of two
   projective lines $\P^1\times\P^1$,
   We write $\mathcal{O}_{X}(a,b)$ for the line bundle of bi-degree $(a,b)$, i.e.
   $$\mathcal{O}_{X}(a,b) = \pi_{V}^{*}(\mathcal{O}_{\PP^1 }(a)) \otimes \pi^{*}_{H}(\mathcal{O}_{\PP^1 }(b)),$$
   where $\pi_{V}$ and $\pi_{H}$ denote the projections on the first (horizontal) and the second (vertical) factor in $\PP^1 \times \PP^1$, respectively.

   It is convenient to work in the present setting with the following definition of symbolic powers.
\begin{definition}
   Let $I \subseteq \field[X]$ be a bi-homogeneous ideal.
   We define the $m$-th symbolic power of $I$ to be the ideal
   $I^{(m)} = \bigcap_{j} P_{i_{j}}$, where $I^{m} = \bigcap_{i} P_{i}$ is a bi-homogeneous primary decomposition,
   and the intersection $\bigcap_{j} P_{i_{j}}$ is taken over all components $P_{i}$ such that the radical $\sqrt{P_{i}}$
   is contained in an associated prime of $I$.
\end{definition}
   A point $P$ in $X$ has coordinates $([a:b],[c:d])$. The ideal defining $P$ as a subscheme of $X$
   is then a prime ideal generated by two forms of bi-degree $(1,0)$ and $(0,1)$ respectively, namely
   by the forms $bx_0-ax_1$ and $dy_0-cy_1$. If $Z=\left\{P_1,\dots,P_s\right\}$ is a finite set of points
   in $X$, then its ideal $I_Z$ is just the intersection $I(Z)=\bigcap\limits_{i=1}^s I(P_i)$.
   It is convenient to denote by $\subm Z$ a \emph{fat points scheme} defined
   by the ideal $I(P_1)^{m_1}\cap\dots\cap I(P_s)^{m_s}$ for an $s$--tuple of positive integers $\subm=(m_1,\dots,m_s)$.
   In this notation $I^{(m)}=I(\subm Z)$ is the ideal defining the fat points scheme $\subm Z$ for
   $\subm=(m,\dots,m)$.
\subsection{Initial degree(s) in bi-homogeneous setting}
   The notion of the initial degree for surfaces with
   Picard number larger than $1$ requires some modifications. The following two variants
   seem the most natural for $X$.
\begin{definition}\label{def:2 alphas}
   Let $I \subset \field[X]$ be a bi-homogeneous ideal. We associate to $I$ the following integers
   $$\alpha^{+}(I) = \min \{k=k_1+k_2 \,\, : \,\, H^{0} (\mathcal{O}_{X}(k_{1}, k_{2}) \otimes I) > 0 \},$$
   $$\alpha^{*}(I) = \min \{ k \,\, : \,\, H^{0}( \mathcal{O}_{X}(k,k) \otimes I) > 0 \}.$$
\end{definition}
   Following Waldschmidt \cite{Wal77} one defines for homogeneous ideals $I\subset\C[\P^n]$
   an asymptotic counterpart of the initial degree
   $$\gamma(I)=\lim\limits_{m\to\infty}\frac{\alpha(I^{(m)})}{m}.$$
   This invariant is called the \emph{Waldschmidt constant} of $I$.
   This notion generalizes verbatim to the multi-homogeneous setting.
\begin{definition}\label{def:walds}
   Let $I \subset \field[X]$ be a non-zero bi-homogeneous ideal. We define the Waldschmidt constant of $I$ as
   $$\gamma^{\bullet}(I) = \lim_{m \longrightarrow \infty} \frac{\alpha^{\bullet}(I^{(m)})}{m},$$
   where $\bullet \in \{*, +\}$.
\end{definition}
   The inclusion $I^{(m)}I^{(k)} \subseteq I^{(m+k)}$ and the Fekete Lemma \cite{Fek} imply in the standard way that the limit in Definition \ref{def:walds} exists,
   and that in fact we have
   $$\gamma^{\bullet}(I) = \inf \frac{\alpha^{\bullet}(I^{(m)})}{m}.$$
\begin{example}
   Let $I(P)$ be the radical ideal defining a point $P$ in $X$. Then $\gamma^{+}(I(P)) = 1$ and $\gamma^{*}(I(P)) = ~\frac{1}{2}$.
\end{example}

\section{Symbolic powers and $\alpha^{*}$ invariant}
\subsection{Configurations of points with no effect of fattening}
   In this section we consider how $\alpha^{*}(I)$ jumps when passing from $I^{(m)}$ to $I^{(m+1)}$.
   In contrary to the projective plane where one has always $\alpha(I^{(m)})<\alpha(I^{(m+1)})$, it might happen on $X$ that
\begin{equation}\label{eq:alphaequal}
   \alpha^{*}(I^{(m)}) = \alpha^{*}(I^{(m+1)}).
\end{equation}
   However the equality in \eqref{eq:alphaequal} is possible only under strong geometric constrains.

   We begin with the following extremely useful Lemma which is \cite[Lemma 2.1]{DST13}
   adopted to the present setting.
\begin{lemma}\label{lem:alphas on subdivisors}
   Let $Z$ be a set of points $P_1,\dots,P_s\in X$
   and
   let $m>n$ be positive integers. Let $I=I(Z)$ and let
   $\beta=\alpha^*(I(mZ))$, $\gamma=\alpha^*(I(nZ))$ and
   $\eps=\beta-\gamma$. Let $C$ be an effective divisor
   of bi-degree $(\beta_1,\beta_2)$ computing $\alpha^*(I(mZ))$,
   i.e. $\beta=\beta_1+\beta_2$.
   Furthermore let
   $$C=C_1+C_2$$
   be a sum of two integral effective non-zero divisors.
   Let $(\beta_1^{(j)},\beta_2^{(j)})$ be the bi-degree of $C_j$ for $j=1,2$ and let $m_i^{(j)}=\ord_{P_i}C_j$
   be the multiplicity of $C_j$ at the point $P_i$.
   We set
   $\submj=(\mj_1,\dots,\mj_s)$ and with
   $$n_i^{(j)}=\max\left\{\mj_i-(m-n),\, 0\right\}$$
   $\subnj=(\nj_1,\dots,\nj_s)$.
   Then
   \begin{itemize}
   \item[i)] $\beta^{(j)}:=\beta_1^{(j)}+\beta_2^{(j)}=\alpha(I(\submj Z))$ and
   \item[ii)] $\alpha^*(I(\submj Z))-\alpha^*(I(\subnj Z))\leq\eps$.
   \end{itemize}
   for $j=1,2$.
\end{lemma}
\proof
   The proof is basically the same as for \cite[Lemma 2.1]{DST13} and we omit it here.
\endproof
\begin{theorem}\label{th:t1}
   Let $Z = \{P_{1}, ..., P_{s}\} \subseteq X$ be a set of points and let $I$ be the radical ideal of $Z$.
   Assume that the condition \eqref{eq:alphaequal} holds for some $m \in \ZZ_{\geq 1}$.
   Then there are finite sets $Z_{V}, Z_{H} \subset \PP^1$ such that $Z = Z_{V} \times Z_{H}$, i.e. $Z$ is a \emph{grid} in $X$.
\end{theorem}
\begin{proof}
   Let $(p,q)$ be a pair of integers such that there exists a section $\sigma \in H^{0}(X, \mathcal{O}_{X}(p,q))$
   vanishing to order at least $m+1$ in points of $Z$ and computing the $\alpha^*$ invariant, i.e.
   $$\alpha^{*}(I^{(m+1)}) = \max\{p, q\}.$$
   Let $D$ be a divisor defined by $\sigma$.
   Let $C$ be a irreducible component of $D$ of bi-degree $(a,b)$ with $mult_{P_{i}} C = m_{i}$ for $i = 1, ..., s$.
   Lemma \ref{lem:alphas on subdivisors} implies now that
\begin{equation}\label{eq:alp+}
   \alpha^{*}\big(I(P_{1})^{m_{1}} \cap ... \cap I(P_{s})^{m_{s}}\big) = \alpha^{*}\big(I(P_{1})^{m_{1}-1} \cap ... \cap I(P_{s})^{m_{s}-1}\big).
\end{equation}
   The Pl\"ucker formula on $\PP^1 \times \PP^1$ implies that
\begin{equation}\label{eq:genus bound}
   ab - a - b + 1 - \sum_{i} {m_{i} \choose 2} \geq 0.
\end{equation}
   On the other hand, by \eqref{eq:alp+} there is no curve of bi-degree $(a-1, b-1)$ with multiplicities $m_{1}-1, ..., m_{s}-1$ through $P_{1}, ..., P_{s}$,
   hence
   $$ab = h^{0}(\mathcal{O}_{X}(a-1,b-1))  \leq \sum_{i} {m_{i} \choose 2} \leq ab - a - b + 1.$$
   This implies that $a=1$ and $b=0$ or $b=1 $ and $a = 0$, hence $C$ is a fiber of a projection in the product $\PP^1 \times \PP^1$.

   Now we exclude the possibility that $D$ consists solely of fibers in one direction e.g. vertical fibers.
   If this were so then removing every fiber from $D$, the multiplicity in every point $P_{i}$ would drop by $1$
   and the bi-degree would drop by the number of removed fibers contradicting \eqref{eq:alphaequal}.
   We conclude that $D$ is the union $k>0$ vertical fibers $V_{1}, ..., V_{k}$ and $l>0$ horizontal fibres $H_{1}, ..., H_{l}$.

   In order to complete the proof we claim that
   \begin{equation}\label{eq:opis Z}
      Z = \{ Q_{ij} = V_{i} \cap H_{j}, i = 1, ..., k, j = 1, ..., l \}.
   \end{equation}

   First we show the $\supseteq$ inclusion. Assume to the contrary that some $Q_{ij}$ is not contained in $Z$.
   Then removing from $D$ the union of fibers $V_{i} \cup  H_{j}$ gives a divisor $D' = D - V_{i} - H_{j}$
   of bi-degree one less than $D$ and vanishing along $Z$ with multiplicities at least $m$. This contradicts the assumption \eqref{eq:alphaequal}.

   Now we show the $\subseteq$ inclusion in \eqref{eq:opis Z}.
   Assume to contrary that there is a point $Q \in Z$, which is not an intersection point
   of one of horizontal fibers $H_1,\dots,H_l$ with one of vertical fibers $V_1,\dots,V_k$.
   Without loss of generality we can assume that $Q \in V_{1}$.
   This implies that $V_{1}$ has multiplicity at least $m+1$ in $D$.
   Removing the union $H_{1} \cup V_{1}$ from $D$ we obtain again a divisor $D' = D - V_{1} - H_{1}$
   of bi-degree one less than that of $D$ vanishing along $Z$ with multiplicities $\geq m$.
   Note that this is indeed the case also in the point $Q_{11}$ (which is the only double point in $H_{1} \cup V_{1}$),
   since in the situation considered here we have $mult_{Q_{11}} (D) \geq m+2$.\\
   The assertion of the Theorem
   follows with the sets $Z_{V} = \pi_{V}(V_{1}\cup ... \cup V_{k})$ and $Z_{H} = \pi_{H}(H_{1} \cup ... \cup H_{l})$.
\end{proof}
   Working still with an $(a,b)$-grid $Z_{V} \times Z_{H}$ (i.e. $\# Z_{V} = a$ and $\# Z_{H} = b$) we show
   that somewhat surprisingly the whole sequence $\alpha^{*}(I^{(m)})$ can be computed explicitly.
   To begin with for $a,b \in \ZZ_{\geq 0}$ we define inductively the following sequence in $\Z^2$:
\begin{equation}\label{eq:sq}
   (a_{m}, b_{m}) = \begin{cases}
   (0,0) &  m = 0, \\
   (a_{m-1}+a,b_{m-1}), & a_{m-1} + a \leq b_{m-1} + b, m \geq 1, \\
   (a_{m-1},b_{m-1}+b), & a_{m-1} + a > b_{m-1} + b, m \geq 1. \\
\end{cases}
\end{equation}
   For this sequence we prove first the following purely numerical lemma.
\begin{lemma}\label{lem:prop1}
   For the sequence $\{(a_{m}, b_{m})\}_{m=0}^{\infty}$ defined in (\ref{eq:sq}) we have
   $a_{m} < b_{m} + b + 1$.
\end{lemma}
\begin{proof}
   Assume to the contrary that
\begin{equation}\label{eq:inseq}
   a_{m} \geq  b_{m} + b + 1.
\end{equation}
   There are the following two possibilities:

   Case $(a)$: If $(a_{m-1}, b_{m-1}) = (a_{m} - a, b_{m})$ , then $a_{m} - a + a > b_{m} + b$, which is a contradiction with \eqref{eq:sq}.

   Case $(b)$: If $(a_{m-1}, b_{m-1}) = (a_{m}, b_{m} - b)$, then $a_{m} + a \geq b_{m} - b + b$.
   Thus
   $$a_{m-1} = a_{m} > b_{m} + b + 1 = (b_{m-1} + b) + b + 1 \geq b_{m-1} + b + 1,$$
   which is \eqref{eq:inseq} with $m$ replaced by $m-1$. Repeating the above argument we arrive to $a_{1} \geq b_{1} + b + 1$.

   Now, for $(a_{1}, b_{1}) = (a,0)$ we have a contradiction with (\ref{eq:sq}). But $(a_{1}, b_{1}) = (0,b)$ contradicts $a_{1} \geq b_{1} + b + 1$.
\end{proof}
\begin{theorem}\label{th:wart}
   Let $I$ be the bi-homogeneous ideal associated with an $(a,b)$-grid of points $Z \subset X$.
   Then for all $m\geq 1$
   $$\alpha^{*}(I^{(m)}) = \max \{a_{m}, b_{m} \},$$
   where $a_{m}, b_{m}$ are defined by (\ref{eq:sq}).
\end{theorem}
\begin{proof}
   Note that for $\alpha_{m} = \frac{a_{m}}{a}$ and $\beta_{m} = \frac{b_{m}}{b}$ the divisor
   $\alpha_{m} \pi^{*}_{V} Z_{v} + \beta_{m} \pi^{*}_{H} Z_{H}$ vanishes at all grid points to order at least $m$,
   hence $\alpha^{*}(I^{(m)}) \leq \max\{a_{m}, b_{m}\}$. We claim that in fact the equality holds.
   Assume without loss of generality that $a_{m} \geq b_{m}$ and assume to the contrary that $\alpha^{*}(I^{(m)}) \leq a_{m} - 1$.

   Let $C$ be a divisor of the bi-degree $(p,q)$ computing $\alpha^{*}(I^{(m)})$, i.e. $\alpha^{*}(I^{(m)})=\max\left\{p,q\right\}$.
   We claim that
\begin{equation}\label{eq:bezout}
   a_{m} - 1 < (m - (\alpha_{m} - 1))b .
\end{equation}
   Taking \eqref{eq:bezout} for granted, let $P$ be a point in $Z_{V}$ and let $V_{P}$ be the fiber over $P$, which is numerically a $(1,0)$-class.
   Intersecting $C$ with $V_{P}$ we have $C \cdot V_{P} = q$. On the other hand, on $V_{P}$ there are $b$ points of $C$ of multiplicity $m$.
   Using \eqref{eq:bezout} and repeatedly Bezout's theorem we see that $V_{P}$ must be a multiplicity $\alpha_{m}$ component of $C$.
   The same argument works for any point in $Z_{V}$ so that $C$ has then at least $\alpha_{m} a = a_{m}$ vertical components counted with multiplicities.
   This contradicts the condition $p \leq a_{m} - 1$.

   Turning to the proof of \eqref{eq:bezout}, note that it follows directly from Lemma \ref{lem:prop1} and the equality $m - \alpha_{m} = \beta_{m}$.
\end{proof}
   It follows immediately from Theorem \ref{th:t1} and Theorem \ref{th:wart} that
   two consecutive equalities as in \eqref{eq:alphaequal} are not possible.
\begin{corollary}
   There is no set of points $Z\subset X$ such that for the ideal $I$ of $Z$
   the equality
   $$\alpha^*(I^{(m+2)})=\alpha^*(I^{(m+1)})=\alpha^*(I^{(m)})$$
   holds for any positive integer $m$.
\end{corollary}
   Of course the same result can be proved along the following lines.
   Let $f$ be an element in $I^{(m+2)}$ of bi-degree $(p,q)$ such that
   $\alpha^*(I^{(m+2)})=\max\left\{p,q\right\}$.
   Taking partial derivatives of $f$ with respect to the first
   set of variables and then with respect to the second set of variables,
   we obtain a polynomial $\wtilde{f}\in I^{(m)}$ of bi-degree $(p-1,q-1)$.
   This shows that there is always inequality
   $\alpha^*(I^{(m+2)})>\alpha^*(I^{(m)})$. Our approach has the advantage
   that it does not call back to differentiation and thus could be
   generalized to arbitrary surfaces.

   It is natural to introduce the following function.
\begin{definition}
   Let  $I$ be a bi-homogenous radical ideal of a set of points $Z$.
   We define the \emph{jump function}
   $$f(Z; m) = \alpha^{*}(I^{(m)}) - \alpha^{*}(I^{(m-1)})$$
   for all $m \in \ZZ_{\geq 1}$ with $\alpha^{*}(I^{(0)}) = 0$.
\end{definition}
   The following
   result is a straightforward consequence of Theorem \ref{th:wart}.
\begin{corollary}
   Assume that $a,b \in \ZZ_{\geq 1}$.
   Working under assumptions of Theorem \ref{th:wart}, the jump function $f(Z; m)$ has infinitely many jumps equal to $0$ and infinitely many jumps
   equal to $\min\{a, b\}$.
\end{corollary}
\begin{proof}
   The idea of the proof is to show that there exists $m$ such that $(a_{m},b_{m}) = (ab, ab)$.
\end{proof}
   The following remark shows that a grid can be recovered from the jump function.
\begin{remark}
   Let $f(Z ;m)$ be a jump function for a certain $(a,b)$-grid in $X$ with $a\leq b$.
   Then $a = f(Z;1)$ and for $r = \min \{j : f(Z;j) < a\}$ the number $b = \sum_{i}^{r} f(Z;i)$.
\end{remark}

\subsection{Configuration of points with the minimal positive effect of fattening.}
   Here we show that the extremely useful result of Chudnovsky \cite{Chu81}
   relating Waldschmidt constants and initial degrees
   generalizes to a multi-homogeneous setting. The original result of Chudnovsky
   is proved with analytic methods. Our approach is modeled on the
   algebraic proof in \cite[Proposition 3.1]{Har--Hun}.
   This result is of independent interest.
\begin{theorem}\label{Chu}
   Let $P_{1}, \ldots, P_{s} \in X$ be mutually distinct points and let
   $I = \bigcap_{i} I(P_{i}) \subset \field[X]$ be the ideal of their union. Then
   $$\frac{\alpha^{*}(I^{(m)})}{m} \geq \frac{ \alpha^{*}(I)}{2}.$$
\end{theorem}
\begin{proof}
   Let $\alpha^{*}(I) = a$. Choose distinct points $Q_{1},\ldots, Q_{t} \in \{P_{1},\ldots, P_{s}\}$ with the smallest possible $t$
   such that $\alpha^{*}(J) = a$ for $J = \bigcap_{j} I(Q_{j})$.
   By minimality of $t$, the points $Q_{j}$ impose independent conditions in bi-degree $(a-1, a-1)$
   so that $t = a^2$. Thus the ideal $J$ is generated in bi-degree $(a,a)$ and hence the only base points of $J_{(a,a)}$ are the points $Q_{j}$.
   In particular, $J_{(a,a)}$ is fixed component free. Let $A$ be a nonzero form in $I_{(b,b)}^{(m)}$, where $b = \alpha^{*}(I^{(m)})$.
   Since $J_{(a,a)}$ is fixed component free, we can choose an element $B \in J_{(a,a)}$ with no common factor with $A$.
   Using Bezout's Theorem adopted to $X$, we have
   $$2 ab = (a,a)\cdot(b,b) = {\rm div}(A) \cdot {\rm div}(B) \geq m a^2,$$
   and hence
   $$\frac{\alpha^{*}(I^{(m)})}{m} \geq \frac{ \alpha^{*}{(I) }}{2}.$$
\end{proof}
   We apply the above Theorem to
   the case with several consecutive jumps of $\alpha^{*}(I^{(m)})$ equal to $1$.
\begin{theorem}\label{thm:jumps by 1}
   Let $I$ be a radical bi-homogeneous ideal of a set $Z=\left\{P_1,\dots,P_s\right\}$
   of points in $X$. Assume that
   \begin{equation}\label{eq:5 jumps by 1}
      \alpha^*(I^{(6)})=\alpha^*(I^{(5)})+1=\ldots=\alpha^*(I^{(2)})+4=\alpha^*(I)+5.
   \end{equation}
   Then $\alpha^*(I)=1$, i.e. $Z$ is contained in a divisor of bi-degree $(1,1)$.\\
   Moreover in order to conclude that $Z$ is contained in a divisor
   of bi-degree $(1,1)$ the sequence of equalities in \eqref{eq:5 jumps by 1} cannot be shortened in general.
\end{theorem}
\proof
   The assumption $\alpha^*(I^{(6)})=\alpha^*(I)+5$ together with Theorem \ref{Chu}
   yields $\alpha^*(I)\leq 2$. If $\alpha^*(I)=1$, then we are done. So it suffices
   to deal with the case $\alpha^*(I)=2$.

   Let $D$ be a divisor of bi-degree $(7,7)$ with multiplicities at least $6$
   in each point $P_i\in Z$. By definition of $\alpha^*$ there is no curve
   of bi-degree $(6,6)$ with this property.

   Now, the idea is to transplant the situation to $\P^2$ via the standard
   birational transformation $\mu:\P^1\times\P^1\to \P^2$, i.e. $\mu$ is the composition
   of the blowing up of a point $S$ such that the horizontal and the vertical fibres
   through $S$ are disjoint from $Z$ followed by blowing down proper transforms of these two fibers
   to points $P$ and $Q$ in $\P^2$.

   On $\P^2$ we have now the following situation. If $D'$ is the proper transform
   of $D$ under $\mu$, then it is a divisor of degree $14$ vanishing to order at least $7$
   at $P$ and at $Q$ and to order at least $6$ in all points in $Z'=\mu(Z)$. We know also
   that there is no divisor of degree $12$ vanishing to order at least $6$ at $P$, $Q$ and
   all points in $Z'$. Obviously $P$, $Q$ and at least one point $R=P_{i_0}\in Z$ are not
   collinear. Applying the standard Cremona transformation $\tau$ based on these points
   to divisor $D'$ results in
   a divisor $D''$ of degree $8$ vanishing to order $6$ at all points in $Z''=\tau(Z'\setminus\left\{R\right\})$.
   Again not all points in $Z''$ can be collinear. On the other hand if
   $F$ is an arbitrary divisor in $\P^2$ and points $A,B,C$ are not collinear, then
   the multi-point Seshadri constant of $\calo_{\P^2}(1)$
   at $A,B,C$ is $1/2$ (see \cite{PSC}
   for definitions and properties of Seshadri constants) and this yields the inequality
   $$\deg(F)\geq\frac12(\mult_{A}F+\mult_{B}F+\mult_CF).$$
   This fact applied to $F:=D''$ gives a contradiction.

   That the result is sharp follows from Example \ref{ex:4 jumps by 1} below.
\endproof
\begin{example}\label{ex:4 jumps by 1}
   For the set $Z$ of points
   $$P_1=([1:0],[1:0]),\; P_2=([1:0],[1:1]),\; P_3=([1:1],[1:0])\; \mbox{ and }\; P_4=([0:1],[1:1])$$
   and $I=I(Z)$ we have
\[
\begin{array}{c|c|c|c|c|c}
  m & 1 & 2 & 3 & 4 & 5 \\ \hline
  \alpha^*(I^{(m)}) & 2 & 3 & 4 & 5 & 6
\end{array}.
\]
\end{example}
\proof
   It is easy to draw divisors with vanishing orders claimed above
   and of the given bi-degree. On the
   other hand, recursive use of Bezout's Theorem excludes
   the possibilities that the bi-degrees could be lower. We leave
   the details to the reader.
\endproof
   It might easily happen that for a radical ideal $I$ of points in $X$
   there are infinitely many jumps by $1$
   in the sequence $\alpha^*(I^{(m)}), m\geq 1$
   and nevertheless $\alpha^*(I)$ is arbitrarily large.
   The following example is an easy consequence of Theorem \ref{th:t1}.
\begin{example}
   Let $a\geq 5$ be an integer and let $Z$ be an $(a,a)$ grid in $X$
   minus a single point.
   Then
   $$\alpha^*(I^{(2k-1)})=ka-1\;\mbox{ and }\; \alpha^*(I^{(2k)})=ka.$$
   In particular $\alpha^*(I)=a-1$.
\end{example}

\section{Symbolic powers and $\alpha^{+}$ invariant}
   The behavior of the $\alpha^+$ invariant is more similar to the initial
   degree in the plane. Let $f$ be a bi-homogeneous polynomial of bi-degree $(a,b)$
   vanishing along a fat point scheme $(m+1)Z$ for some set of points $Z\subset X$.
   Taking a partial derivative of $f$ with respect to the first or the second set
   of variables, we obtain a non-zero polynomial of bi-degree $(a-1,b)$ or $(a,b-1)$
   vanishing along $mZ$. This shows that there is always the strong inequality
   $$\alpha^+(I^{(m)})<\alpha^+(I^{(m+1)})$$
   for all $m\geq 1$.\\
   We describe now the situation when the effect of fattening is the minimal
   possible, i.e. equal to $1$.
\begin{theorem}\label{thm:jumps by 1 alpha plus}
   Let $Z = \{P_{1}, ..., P_{s}\} \subseteq X$ be a set of points and let $I$ be the radical ideal of $Z$. Assume that
   \begin{equation}\label{eq:diffone}
      \alpha^{+}(I^{(m+1)}) = \alpha^{+}(I^{(m)}) + 1
   \end{equation}
   for some $m \in \ZZ_{\geq 1}$. Then all points $\{P_{1}, ..., P_{s}\}$ lie on a single vertical or horizontal fiber.
\end{theorem}
\proof
   The assertion is trivial for $s=1$, so we assume $s \geq 2$.
   The first part of the proof is quite analogous to that of Theorem \ref{th:t1}.
   Let $\sigma \in H^{0}(X, \mathcal{O}_{X}(p,q))$ be a section vanishing to order at least $m+1$ in points in $Z$
   with $(p,q)$ minimal, i.e.
   $$\alpha^{+}(I^{(m+1)}) = p + q = d$$
   Let $D$ be a divisor defined by $\sigma$.
   Let $C$ be a irreducible component of $D$ of bi-degree $(a,b)$ and with $mult_{P_{i}} C = m_{i}$.
   A simple modification of \cite[Lemma 2.1]{DST13} implies that
   $$\alpha^{+}\big(I(P_{1})^{m_{1}} \cap ... \cap I(P_{s})^{m_{s}}\big) =
      a+b =
      \alpha^{+}\big(I(P_{1})^{m_{1}-1} \cap ... \cap I(P_{s})^{m_{s}-1}\big) + 1.
   $$
   In particular, there is no curve of bi-degree $(k, a+b-2 - k)$,
   where $k$ is an arbitrary non-negative integer less or equal than $a+b-2$,
   with multiplicities $m_{1}-1, ..., m_{s}-1$ through $P_{1}, ..., P_{s}$. Hence
   $$h^{0}(\mathcal{O}_{X}(k , a+b-2 - k)) = (k+1)(a+b-1-k) \leq \sum_{i} {m_{i} \choose 2} \leq ab - a - b + 1,$$
   where the last inequality is the genus bound \eqref{eq:genus bound}.
   Since the above inequality holds for all $k$'s, we can take $k = a - 1$.
   Then
   $$ab \leq ab - a - b + 1.$$
   This implies that $a=1$ and $b=0$ or $b=1 $ and $a = 0$, hence $C$ is a fiber of a projection in the product $\PP^1 \times \PP^1$.

   In the next step we exclude the possibility that $D$ has two or more
   vertical or horizontal components. Suppose to the contrary that
   $V$ and $W$ are two vertical fibers contained in the support of $D$.
   Then $D'=D-(V+W)$ vanishes at all points in $Z$ to order at least $m$
   (since we remove smooth disjoint components). This contradicts however \eqref{eq:diffone}.
   The proof if there are two horizontal components is completely
   analogous.

   It remains to exclude the possibility that $D$ is supported
   on a union $V+H$ of a vertical fiber $V$ and a horizontal fiber $H$.
   Let $Q=H\cap V$. Since $s\geq 2$, there must be some other point $R\in Z$
   contained either in $V$ or in $H$. Then either $(m+1)V$ or $(m+1)H$
   is contained in $D$, so that in particular $\mult_QD\geq m+2$.
   Thus for $D'=D-(V+H)$ we have $\mult_PD'\geq m$ for all points $P\in Z$.
   This contradicts \eqref{eq:diffone} again and we are done.
\endproof
   Turning to the jumps of the $\alpha^+$ invariant by $2$,
   the above proof yields immediately the following observation.
\begin{corollary}
   Let $Z = \{P_{1}, ..., P_{s}\} \subseteq X$ be a set of points and let $I$ be the radical ideal of $Z$. Assume that
   \begin{equation}\label{eq:diff two}
      \alpha^{+}(I^{(m+1)}) = \alpha^{+}(I^{(m)}) + 2
   \end{equation}
   for some $m \in \ZZ_{\geq 1}$. Let $D$ be a divisor on $X$ of bi-degree $(p,q)$
   such that $\alpha^{+}(I^{(m+1)})=p+q$. Then any irreducible component
   of $D$ has bi-degree $(1,0)$, or $(1,1)$ or $(0,1)$. Moreover
   there are at most two horizontal or vertical components.
\end{corollary}
\proof
   Let $C$ be a component of $D$ of bi-degree $(a,b)$. Arguing as in the proof
   of Theorem \ref{thm:jumps by 1 alpha plus} and keeping the notation from that proof,
   we obtain that the inequality
   $$(k+1)(a+b-2-k) \leq  ab - a - b + 1$$
   must hold for all $k$ in the range $0,\dots,a+b-1$.
   For $k=a-1$ we obtain
   $$a(b-1)\leq  ab - a - b + 1$$
   which implies $b\leq 1$. Analogously, for $k=b-1$ we get $a\leq 1$. This
   leaves only three bi-degrees possible for $C$.

   Turning to the second claim, assume that there were at least $3$ vertical
   components in $D$. Then removing them from $D$ lowers the multiplicities
   at all points $P_i$ at most by $1$ but the sum of degrees by at least $3$
   contradicting \eqref{eq:diff two}. The same argument works for
   horizontal fibers.
\endproof
   We conclude with the example showing that all types of components
   listed in the above Corollary can occur simultaneously. For simplicity
   all coordinates are affine, a point $(a,b)$ is the point $((1:a),(1:b))$ in
   $X$.
\begin{example}
   Consider the following set of points
   $P_1=(0,0)$, $P_2=(1,1)$, $P_3=(1,2)$, $P_4=(2,2)$, $P_5=(3,0)$, $P_6=(3,3)$
   and let $I$ be the bi-homogeneous ideal of their union. Then
   \begin{equation}\label{eq:46}
      \alpha^+(I)=4\;\mbox{ and }\; \alpha^+(I^{(2)})=6.
   \end{equation}
\end{example}
\proof
   The following figures show a divisor computing $\alpha^+(I)$ (there are many
   possibilities) and the divisor computing $\alpha^+(I^{(2)})$ (this divisor
   is unique). The diagonal line is a divisor of bi-degree $(1,1)$.
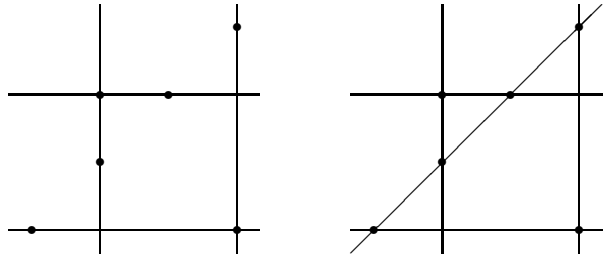
\begin{center}
\unitlength.15mm
\begin{picture}(520,220)(0,0)
\put(80,0){\line(0,1){220}}   
\put(0,20){\line(1,0){220}}   
\put(0,140){\line(1,0){220}}   
\put(200,0){\line(0,1){220}}   
%
 \put(20, 20){\circle*{7}}
 \put(80, 80){\circle*{7}}
 \put(80, 140){\circle*{7}}
 \put(140, 140){\circle*{7}}
 \put(200, 20){\circle*{7}}
 \put(200,200){\circle*{7}}
\put(380,0){\line(0,1){220}}   
\put(300,20){\line(1,0){220}}   
\put(300,140){\line(1,0){220}}   
\put(500,0){\line(0,1){220}}   
\put(300,0){\line(1,1){220}}   
 \put(320, 20){\circle*{7}}
 \put(380, 80){\circle*{7}}
 \put(380, 140){\circle*{7}}
 \put(440, 140){\circle*{7}}
 \put(500, 20){\circle*{7}}
 \put(500,200){\circle*{7}}
\end{picture}
\end{center}
   Thus we have exhibited explicitly divisors realizing
   $\leq $ inequalities in \eqref{eq:46}.
   We leave it to a motivated reader to check that neither for $I$
   nor for $I^{(2)}$ the $\alpha^+$  invariant can be lowered.
\endproof

\paragraph*{\emph{Acknowledgement.}}
   This note resulted from a workshop on linear series held in Lanckorona
   in February 2013. We thank Pedagogical University of Cracow for the financial
   support.

\bigskip \small

\bigskip
   Marcin Dumnicki, Halszka Tutaj-Gasi\'nska,
   Jagiellonian University, Institute of Mathematics, {\L}ojasiewicza 6, PL-30-348 Krak\'ow, Poland

\bigskip
   Magdalena Baczy\'nska, Agata Habura, Grzegorz Malara, Piotr Pokora, Tomasz Szemberg, Justyna Szpond,
   Pedagogical University of Cracow, Institute of Mathematics,
   Podchor\c a\.zych 2,
   PL-30-084 Krak\'ow, Poland

\end{document}